\documentclass[11pt]{article}

\usepackage{geometry}
\usepackage{fullpage}

\usepackage[english]{babel}
\usepackage[utf8]{inputenc}
\usepackage[T1]{fontenc}
\usepackage{dsfont}

\usepackage{psfrag}

\usepackage{calrsfs}
\usepackage{mathrsfs}
\usepackage{pdfsync}
\usepackage{amsmath, amsthm, amssymb, amsfonts}
\usepackage[shortlabels]{enumitem}
\usepackage{url}
\usepackage{float}

\usepackage[usenames]{color}

\usepackage{tikz}

\usetikzlibrary{arrows,decorations.pathmorphing,backgrounds,positioning,fit,automata}

\usetikzlibrary{arrows}
\usetikzlibrary{petri}
\usetikzlibrary{topaths}

\usepackage{indentfirst,calc,euscript}
\usepackage{setspace}
\usepackage[reals]{layout}
\usepackage{xr}
\usepackage{amscd}

\usepackage{sgame}
\usepackage{subfigure}

\usepackage[colorlinks=true,breaklinks=true,bookmarks=true,urlcolor=blue,
     citecolor=blue,linkcolor=blue,bookmarksopen=false,draft=false]{hyperref}

\newcommand{\ignore}[1]{}

\newtheorem{theorem}{Theorem}

\newtheorem{proposition}[theorem]{Proposition}

\theoremstyle{definition}

\newtheorem{remark}[theorem]{Remark}

\newtheorem{case}{Case}

\numberwithin{equation}{section}

\numberwithin{theorem}{section}

\newcommand{\m}{\mathbb}

\author{Bruno Ziliotto\thanks{CNRS, Paris Dauphine University, PSL, Place du Mar\'echal de Lattre de Tassigny, 75016 Paris,
France. \newline E-mail: ziliotto@math.cnrs.fr}}
\title{Convergence of the solutions of the discounted Hamilton-Jacobi equation: a counterexample}
\begin{document}
\maketitle
\bibliographystyle{plain}
\begin{abstract}
This paper provides a counterexample about the asymptotic behavior of the solutions of a discounted Hamilton-Jacobi equation, as the discount factor vanishes. The Hamiltonian of the equation is a 1-dimensional continuous and coercive Hamiltonian. 
\end{abstract}
\section{Introduction and main result}
Let $n \geq 1$. Denote by $\m{T}^n=\m{R}^n/\m{Z}^n$ the $n$-dimensional torus. For $c \in \m{R}$, consider the Hamilton-Jacobi equation
\begin{equation*}
H(x,Du(x))=c \quad (E_0)
\end{equation*}
where the Hamiltonian $H: \m{T}^n \times \m{R}^n \rightarrow \m{R}$ is jointly continuous and coercive in the momentum. In order to build solutions of the above equation, Lions, Papanicolaou and Varadhan \cite{LPV86} have introduced a technique called 
\textit{ergodic approximation}. For $\lambda \in (0,1]$, consider the discounted Hamilton-Jacobi equation
\begin{equation} \label{disceq}
\lambda v_\lambda(x)+H(x,Dv_\lambda(x))=0   \quad (E_\lambda)
\end{equation}
%where the Hamiltonian $H: \m{T}^n \times \m{R}^n \rightarrow \m{R}$ is continuous and coercive in the momentum. 
By a standard argument, this equation has a unique viscosity solution $v_{\lambda}:\m{T}^n \rightarrow \m{R}$. Moreover, $(-\lambda v_{\lambda})$ converges uniformly as $\lambda$ vanishes to a constant $c(H)$ called the \textit{critical value}. Set $u_{\lambda}:=v_{\lambda}+c(H)/\lambda$. 
%Lions, Papanicolaou and Varadhan \cite{LPV86} have proved that there exists a unique constant $c(H)$, called the \textit{critical value}, such that the following equation 
%\begin{equation*}
%H(x,Du(x))=c(H) \quad (E_0)
%\end{equation*}
%has a solution. For $c=c(H)$, consider the family $(u_{\lambda})$ of $(E_{\lambda})$. 
The family $(u_\lambda)$ is equi-Lipschitz, and converges uniformly along subsequences towards a solution of $(E_0)$, for $c=c(H)$. Note that $(E_0)$ may have several solutions. Recently, under the assumption that $H$ is convex in the momentum, Davini, Fathi, Iturriaga and Zavidovique \cite{DFIZ16} have proved that $(u_\lambda)$ converges uniformly (towards a solution of ($E_0$)). In addition, they proved that the solution can be characterized using Mather measures and Peierls barriers. 
\\
Without the convexity assumption, the question of whether $(u_\lambda)$ converges or not remained open. This paper solves negatively this question and provides a 1-dimensional continuous and coercive Hamiltonian for which $(u_\lambda)$ does not converge\footnote{Note that for time-dependent Hamilton-Jacobi equations, several counterexamples about the asymptotic behavior of solutions have been pointed out in \cite{BS00}.}.
\begin{theorem} \label{main}
There exists a continuous Hamiltonian $H: \m{T}^1 \times \m{R} \rightarrow \m{R}$ that is coercive in the momentum, such that $u_\lambda$ does not converge as $\lambda$ tends to 0. 
\end{theorem}
The example builds on a class of discrete-time repeated games called \textit{stochastic games}. The main ingredient is to establish a connection between recent counterexamples to the existence of the limit value in stochastic games (see \cite{vigeral13, Z13}) and the Hamilton-Jacobi problem\footnote{Let us mention the work \cite{KS06,KS10,IS11,Z16} as other illustrations of the use of repeated games in PDE problems.}.
\\
The remainder of the paper is structured as follows. Section 2 presents the stochastic game example. Section 3 shows that in order to prove Theorem \ref{main}, it is enough to study the asymptotic behavior of the stochastic game, when the discount factor vanishes. Section 4 determines the asymptotic behavior of the stochastic game. 
\section{The stochastic game example}
Given a finite set $A$, the set of probability measures over $A$ is denoted by $\Delta(A)$. Given $a \in A$, the Dirac measure at $a$ is denoted by $\delta_a$. 
\subsection{Description of the game}
%Let $c_0:=2-\sqrt{2}$. 
Consider the following stochastic game $\Gamma$, described by:
\begin{itemize}
\item
A state space $K$ with two elements $\omega_{1}$ and $\omega_{-1}$: $K=\left\{\omega_1,\omega_{-1} \right\}$,
\item
An action set $I=\left\{0,1\right\}$ for Player 1,
\item
An action set $J= \left\{2-\sqrt{2}+2^{-2n}, n \geq 1 \right\} \cup \left\{2-\sqrt{2} \right\}$ for Player 2,
\item
For each $(k,i,j) \in K\times I \times J$, a transition $q(. \, | k,i,j) \in \Delta(K)$ defined by:
\begin{eqnarray*}
q(.\,|\omega_1,i,j)&=&[i j+(1-i)(1-j)] \delta_{\omega_1}+[i(1-j)+(1-i)j] \delta_{\omega_{-1}},
\\
q(.\,|\omega_{-1},i,j)&=&[i (1-j)+(1-i)j] \delta_{\omega_1}+[i j+(1-i)(1-j)] \delta_{\omega_{-1}}. 
\end{eqnarray*}
\item
A payoff function $g: K \times I \times J \rightarrow [0,1]$, defined by
\begin{equation*}
g(\omega_1,i,j)=i j+2(1-i)(1-j) \quad \text{and} \quad g(\omega_{-1},i,j)=-i j-2(1-i)(1-j).
\end{equation*}
\end{itemize}
Let $k_1 \in K$. The stochastic game $\Gamma^{k_1}$ starting at $k_1$ proceeds as follows:
\\
\begin{itemize}
\item
The initial state is $k_1$. At first stage, Player 2 chooses $j_1 \in J$ and announces it to Player 1. Then, Player 1 chooses $i_1 \in I$, and announces it to Player 2. The payoff at stage 1 is $g(k_1,i_1 ,j_1)$ for Player 1, and $-g(k_1,i_1,j_1)$ for Player 2. A new state $k_2$ is drawn from the probability $q(. \ | k_1,i_1,j_1)$ and announced to both players. Then, the game moves on to stage 2.
\item
At each stage $m \geq 2$, Player 2 chooses $j_m \in J$ and announces it to Player 1. Then, Player 1 chooses $i_m \in I$, and announces it to Player 2. The payoff at stage $m$ is $g(k_m,i_m,j_m)$ for Player 1, and $-g(k_m,i_m,j_m)$ for Player 2. A new state $k_{m+1}$ is drawn from the probability $q(. \ | k_m,i_m,j_m)$ and announced to both players. Then, the game moves on to stage $m+1$. 
\end{itemize}
\begin{remark}
The action set of Player 2 can be interpreted as a set of randomized actions. Indeed, imagine that Player 2 has only two actions, $1$ and $0$. These actions are called \textit{pure actions}. At stage $m$, if Player 2 chooses $j_m \in J$, this means that he plays $1$ with probability $j_m$, and 
$0$ with probability $1-j_m$. Denote by $\widetilde{j_m} \in \left\{0,1\right\}$ his realized action. Player 1 knows $j_m$ before playing, but does not know $\widetilde{j_m}$. If Player 1 chooses $i_m \in I$ afterwards, then the realized payoff is $g(k_m,i_m,\widetilde{j_m})$. 
Thus, the payoff $g(k_m,i_m,j_m)$ represents the expectation of $g(k_m,i_m,\widetilde{j_m})$. Likewise, the transition $q(. \, | k_m,i_m,j_m)$ represents the law of $q(k_m,i_m,\widetilde{j_m})$. The transition and payoff in $\Gamma$ when players play pure actions can be represented by the following matrices: 
 \begin{table}[H]
\label{game}
\centering
\caption{Transition and payoff functions in state $\omega_1$ and $\omega_{-1}$}
   \vspace{0.4cm}
\begin{tabular}{ |l|*{3}{c|}}
\hline 
$\omega_1$
& 1 & 0 \\\hline
1 & $1$ & $\overrightarrow{0}$ \\\hline
0 & $\overrightarrow{0}$ & $2$ \\\hline
\end{tabular}
\hspace{1cm}
\begin{tabular}{|l|*{3}{c|}}
\hline %\backslashbox{J1}{J2}
$\omega_{-1}$
&  1 & 0 \\\hline
1 & $-1$ & $\overleftarrow{0}$ \\\hline
0 & $\overleftarrow{0}$ & $-2$ \\\hline
\end{tabular}
\end{table}
The left-hand side matrix stands for state $\omega_1$, and the right-hand side matrix stands for state $\omega_{-1}$. Consider the left-hand side matrix. Player 1 chooses a row (either $1$ or $0$), and Player 2 chooses a column (either $1$ or $0$). The payoff is given by the numbers: for instance, $g(1,1)=1$ and $g(1,0)=0$. The arrow means that when the corresponding actions are played, the state moves on to state $\omega_{-1}$; otherwise, it stays in $\omega_1$. For instance, $q(.|\omega_1,1,1)=\delta_{\omega_1}$ and $q(.|\omega_1,1,0)=\delta_{\omega_{-1}}$. The interpretation is the same for the right-hand side matrix. In the game $\Gamma$, Player 1 can play only pure actions (1 or 0), and Player 2 can play $1$ with some probability $j \in J$. 
\\
This matrix representation is convenient to understand the strategic aspects of the game. 
\end{remark}
\vspace{0.3cm}
Let us now define formally \textit{strategies}. 
In general, the decision of a player at stage $m$ may depend on all the information he has: that is, the stage $m$, and all the states and actions before stage $m$. In this paper, it is sufficient to consider a restricted class of strategies, called \textit{stationary strategies}. Formally, a stationary strategy for Player 1 is defined as a mapping 
 $y:K \times J \rightarrow I$. The interpretation is that at stage $m$, if the current state is $k$, and Player 2 plays $j$, then Player 1  plays $y(k,j)$. Thus, Player 1 only bases his decision 
 on the current state and the current action of Player 2. Denote by $Y$ the set of stationary strategies for Player 1. 
 \\
 A stationary strategy for Player 2 is defined as a mapping 
 $z:K \rightarrow J$. The interpretation is that at stage $m$, if the current state is $k$, then Player 2  plays $z(k)$. Thus, Player 2 only bases his decision 
 on the current state. Denote by $Z$ the set of stationary strategies for Player 2. 
 \\
The sequence $(k_1,i_1,j_1,k_2,i_2,j_2,...,k_m,i_m,j_m,...) \in H_\infty:=(K \times I  \times J)^{\m{N}^*}$ generated along the game is called \textit{history} of the game. Due to the fact that state transitions are random, this is a random variable. The law of this random variable depends on the initial state $k_1$ and the pair of strategies $(y,z)$, and is denoted by $\mathbb{P}^{k_1}_{y,z}$. 
 \\
%A pair of strategies $(x,y)$ induces a unique probability measure on $H_\infty$, denoted by $\mathbb{P}^{k_1}_{\sigma,\tau}$. 
We will call $g_m$ the $m$-stage random payoff $g(k_m,i_m,j_m)$. 
Let $\lambda \in (0,1]$. The game $\Gamma^{k_1}_\lambda$ is the game where the strategy set of Player 1 (resp. 2) is $Y$ (resp. $Z$), and the payoff is %defined by its normal form $(\Sigma,\mathcal{T},
$\gamma_{\lambda}^{k_1}$, where 
\begin{equation*}
\gamma^{k_1}_{\lambda}(y,z)=\mathbb{E}^{k_1}_{y,z}\left(\sum_{m \geq 1} (1-\lambda)^{m-1} g_m \right).
\end{equation*}
The goal of Player 1 is to maximize this quantity, while the goal of Player 2 is to minimize this quantity. 
%Note that the optimal choice of Player 1 depends on Player 2's choice, and reciprocally. To avoid this circular problem, we consider that Player 2 chooses first his strategy, and announced it to Player 1; then, Player 1 chooses his strategy. 
The game $\Gamma^{k_1}_\lambda$ has a value, that is:
\begin{equation*}
\min_{z \in Z} \max_{y \in Y} \gamma^{k_1}_{\lambda}(y,z)=\max_{y \in Y} \min_{z \in Z} \gamma^{k_1}_{\lambda}(y,z).
\end{equation*}
The value of $\Gamma^{k_1}_\lambda$ is then defined as the above quantity, and is denoted by $w_{\lambda}(k_1)$. A strategy for Player 1 is \textit{optimal} if it achieves the right-hand side maximum, and a strategy for Player 2 is \textit{optimal} if it achieves the left-hand side minimum. 
%Define the \textit{value} of $\Gamma^{k_1}_\lambda$ as being the real number $v_{\lambda}(k_1) \in \m{R}$ defined by
%%The game $\Gamma^{k_1}_\lambda$ has a value, called $v_{\lambda}(k_1)$:
%\begin{equation*}
%v_{\lambda}(k_1)=\min_{b \in J^2} \max_{a \in I^2} \gamma^{k_1}_{\lambda}(a,b).
%%=\min_{\sigma \in \Sigma} \max_{\tau \in \mathcal{T}} \gamma^{k_1}_{\lambda}(\sigma,\tau)
%\end{equation*}
The interpretation is that if players are rational they should play optimal strategies, and as a result Player 1 should get $w_{\lambda}(k_1)$, and Player 2 should get $-w_{\lambda}(k_1)$. 
\subsection{Asymptotic behavior of the discounted value} \label{asympt}
As we shall see in the next section, for each $\lambda \in (0,1]$, one can associate a discounted Hamilton-Jacobi equation with $c(H)=0$, such that its solution evaluated at $x=1$ is approximately $w_\lambda(\omega_1)$, for $\lambda$ small enough. Thus, the asymptotic behavior of this quantity needs to be studied. 
\\
\\
Define $\lambda_n:=\displaystyle 2^{-2n}\left(\frac{3}{4}-\frac{1}{\sqrt{2}} \right)^{-1}$ and $\mu_n:=\displaystyle 2^{-2n-1}\left(\frac{3}{4}-\frac{1}{\sqrt{2}} \right)^{-1}$. 
\begin{proposition} \label{propsto}
The following hold: 
\begin{enumerate}[(i)]
%\item \label{prop1}
%$0 \leq v_{\lambda}(\omega_{1}) \leq 2 \lambda$
\item \label{prop2}
$w_{\lambda}(\omega_{-1}) \leq w_{\lambda}(\omega_1) \leq w_{\lambda}(\omega_{-1})+2$
%\item \label{prop3}
%\begin{equation*} 
%\lim_{\lambda \rightarrow 0} w_{\lambda}(\omega_1)=\lim_{\lambda \rightarrow 0} w_{\lambda}(\omega_{-1})=0 
%\end{equation*}
\item \label{key}
$\lim_{n \rightarrow +\infty} w_{\lambda_n}(\omega_1)=1/\sqrt{2}$ and 
$\liminf_{n \rightarrow +\infty} w_{\mu_n}(\omega_1) >1/\sqrt{2}$. Consequently, $(w_{\lambda}(\omega_1))$ does not have a limit when $\lambda \rightarrow 0$. 
\end{enumerate}
\end{proposition}
The proof of the above proposition is done in Section \ref{proofprop}. As far as the proof of Theorem \ref{main} is concerned, the key point is \ref{key}. is Let us give here some piece of intuition for this result. Consider the game $\Gamma'$ that is identical to $\Gamma$, except that
Player 2's action set is $[0,1]$ instead of $J$. For each $\lambda \in (0,1]$, denote by $w'_{\lambda}$ its discounted value. Because $J \subset [0,1]$, Player 2 is better off in the game $\Gamma'$ compared to the game $\Gamma$: $w'_\lambda \leq w_\lambda$. Interpret now $\Gamma$ and $\Gamma'$ as games with randomized actions, as in Table \ref{game}. 
As $\lambda$ vanishes, standard computations show that an (almost) optimal stationary strategy for Player 2 in $\Gamma'^{\omega_1}_{\lambda}$ is to play $1$ with probability 
$p^*(\lambda):=2-\sqrt{2}+\left(\frac{3}{4}-\frac{1}{\sqrt{2}}\right) \lambda$ in both states $\omega_1$ and $\omega_{-1}$, and $(w_{\lambda}(\omega_1))$ converges to $\frac{1}{\sqrt{2}}$. %and $w_{\lambda}(\omega_{-1})$ converges to $-\frac{1}{\sqrt{2}}$. 
%Consequently,
%$\lim_{\lambda \rightarrow 0} w_{\lambda}(\omega_1)=\lim_{\lambda \rightarrow 0} w_{\lambda}(\omega_{-1})=0$.
\\
Moreover, for all $n \geq 1$, $p^*(\lambda_n) \in J$. Thus, this strategy is available for Player 2 in $\Gamma$, and consequently $w_{\lambda_n}(\omega_1)=w'_{\lambda_n}(\omega_1)+O(\lambda_n)$, as $n$ tends to infinity. 
\\
On the other hand, for all $n \geq 1$, $p^*(\mu_n) \notin J$, and the distance of $p^*(\mu_n)$ to $J$ is larger than $\left(\frac{3}{4}-\frac{1}{\sqrt{2}}\right) \mu_n/2$. Consequently, the distance of the optimal strategy in $\Gamma^{\omega_1}_{\mu_n}$ to the optimal strategy in $\Gamma'^{\omega_1}_{\mu_n}$ is of order $\mu_n$. This produces a payoff difference of order $\mu_n$ at each stage, and thus of order 1 in the whole game. Thus,  Player 2 is significantly disadvantaged in $\Gamma^{\omega_1}_{\mu_n}$ compared to $\Gamma'^{\omega_1}_{\mu_n}$, and the difference between $w_{\mu_n}(\omega_1)$ and $w'_{\mu_n}(\omega_1)$ is of order 1. %, and thus a difference of order 1 between $v_{\mu_m}(\omega_1)/\mu_m$ and $v'_{\mu_m}(\omega_1)/\mu_m$. 
\\

\begin{remark}
As we shall see in the following section, we have $\lim_{\lambda \rightarrow 0} \lambda w_{\lambda}(\omega_1)=\lim_{\lambda \rightarrow 0} \lambda w_{\lambda}(\omega_{-1})=0$. 

\end{remark}

The next section explains how to derive the counterexample and Theorem \ref{main} from Proposition \ref{propsto}.  
\section{Link with the PDE problem and proof of Theorem \ref{main}}
%In this subsection, we assume that $K=\left\{\omega_1,\omega_{-1} \right\}$. 
The following proposition expresses $w_\lambda$ as the solution of a functional equation called \textit{Shapley equation}. 
\begin{proposition} \label{shapley}
Let $\lambda \in (0,1]$ and $u_\lambda:=(1+\lambda)^{-1}w_{\lambda/(1+\lambda)}$. 
For each $r \in \left\{-1,1\right\}$, the two following equations hold: 
\begin{enumerate}[(i)]
\item \label{shapleyv}
\begin{eqnarray*}
w_\lambda(\omega_r)&=&\min_{j \in J} \max_{i \in I} \left\{ g(\omega_r,i,j)+(1-\lambda) 
 \left[q(\omega_{r} | \omega_r,i,j) w_\lambda(\omega_r) +q(\omega_{-r} | \omega_r,i,j) w_{\lambda}(\omega_{-r})\right] \right\}
%\\
%&=&\max_{i \in I} \min_{j \in J} \left\{\lambda g(k_1,i,j)+(1-\lambda) 
%\sum_{(i,j) \in  I \times J} \sigma(i) \tau(j) \left[q(k_1,i,j)\right]^{k_2} \left[v_\lambda(k_2)-v_\lambda(k_1) \right]\right\}+(1-\lambda) v_{\lambda}(k_1)
\end{eqnarray*}
\item \label{shapleyu}
\begin{equation*}
\lambda u_\lambda(\omega_r)=\min_{j \in J} \max_{i \in I} \left\{ g(\omega_r,i,j)+
q(\omega_{-r}|\omega_r,i,j) \left[u_\lambda(\omega_{-r})-u_\lambda(\omega_r) \right]\right\}
\end{equation*}
%\item
%\begin{equation*}
%\lambda u_\lambda(\omega_r)=\min_{j \in J} \max_{i \in I, z \in \left\{-1,0,1\right\}} \left\{ g(\omega_r,i,j)-10 \left|z\right|+
%\left(\left[q(\omega_r,i,j)\right]^{\omega_{-r}}+z \right) \left[u_\lambda(\omega_{-r})-u_\lambda(\omega_r) \right]\right\}-C
%\end{equation*}
\begin{proof}
\begin{enumerate}
\item
The intuition is the following. Consider the game $\Gamma^{\omega_r}_{\lambda}$. 
At stage 1, the state is $\omega_r$. The term $g$ represents the current payoff, and the term $(1-\lambda)[...]$ represents the future optimal payoff, that is, the payoff that Player 1 should get from stage 2 to infinity. Thus, this equation means that the value of $\Gamma^{\omega_r}_{\lambda}$ coincides with the value of the one-stage game, where the payoff is a combination of the current payoff and the future optimal payoff. For a formal derivation of this type of equation, we refer to \cite[VII.1., p. 392]{MSZ}.
%Let $N \geq 1$. Consider the stochastic game $\Gamma^N$ where the state space is $K$, the action set of Player 1 is $\left\{0,1\right\}^N$, the action set of Player 2 is $\left\{0,1\right\}$, and the payoff function is 
%\begin{equation*}
%g^N(k,i,j):= \left\{
%\begin{array}{ll}
%\displaystyle g(k,i(j),j) & \mbox{if} \ j \leq N-1, \\
%g(k,i(N),j) & \mbox{if} \ \ j \geq N,
%\end{array}
%\right.
%\end{equation*}
%and the transition function is
%\begin{equation*}
%q^N(k,i,j):= \left\{
%\begin{array}{ll}
%\displaystyle q(k,i(j),j) & \mbox{if} \ j \leq N-1, \\
%q(k,i(N),j) & \mbox{if} \ \ j \geq N.
%\end{array}
%\right.
%\end{equation*}
%This stochastic game has 
%g^N(k,i,j):=g(k,i(j),j)
%\end{equation*}
\item
Evaluating the previous equation at $\lambda/(1+\lambda)$ yields
\begin{eqnarray*}
w_{\frac{\lambda}{1+\lambda}}(\omega_r)&=&\min_{j \in J} \max_{i \in I} \left\{g(\omega_r,i,j)+ 
 \frac{1}{1+\lambda}[q(\omega_{r} | \omega_r,i,j) w_{\frac{\lambda}{1+\lambda}}(\omega_r) +q(\omega_{-r}|\omega_r,i,j) w_{\frac{\lambda}{1+\lambda}}(\omega_{-r})] \right\}
\end{eqnarray*}
Using the fact that $q(\omega_{r} | \omega_r,i,j)=1-q(\omega_{-r}|\omega_r,i,j)$ yields the result.

 % from between If from stage 2, Player 1 plays a strategy that is optimal strategy Players know that from stage 2, they will be able to play optimally in $\Gamma^{\omega_r}_{\lambda}$
\end{enumerate}
\end{proof}
%In a similar fashion, 
%\begin{equation*}
%\lambda u_\lambda(\omega_{-1})=\min_{j \in J} \max_{i \in I} \left\{ g(\omega_{-1},i,j)+
%\left[q(\omega_{-1},i,j)\right]^{\omega_1} \left[u_\lambda(\omega_1)-u_\lambda(\omega_{-1}) \right]\right\}-C
%\end{equation*}
\end{enumerate}
\end{proposition}
%For $k_1 \in K$, define $a_1(\sigma,\tau) \in \Delta(K)$ by
%\begin{equation}
%a_1(\sigma,\tau):=\sum_{(i,j) \in I \times J} \sigma(i) \tau(j) \left[q(\omega_1,i,j)\right]^{\omega_{-1}},
%\end{equation}
%and
%\begin{equation}
%a_2(\sigma,\tau):=\sum_{(i,j) \in I \times J} \sigma(i) \tau(j) \left[q(\omega_{-1},i,j)\right]^{\omega_1},
%\end{equation}
%\begin{equation}
%H_r(p):=-\max_{x \in \Delta(I)} \min_{y \in \Delta(J)} \left\{ g(k,i,j)+ a_r(x,y) \cdot p  \right\}.
%\end{equation}
For $p \in \m{R}$, define
$H_1:\m{R} \rightarrow \m{R}$ and $H_{-1}:\m{R} \rightarrow \m{R}$ by
\begin{equation*}
H_1(p):= \left\{
\begin{array}{ll}
\displaystyle 
-\min_{j \in J} \max_{i \in I}
\left\{ g(\omega_1,i,j)-p \cdot
([i(1-j)+(1-i)j] \right\},
 & \mbox{if} \ |p| \leq 2, \\
H_1\left(2\frac{p}{|p|}\right)+|p|-2 & \mbox{if} \  |p| > 2.
\end{array}
\right.
\end{equation*}
\begin{equation*}
H_{-1}(p):= \left\{
\begin{array}{ll}
\displaystyle 
-\min_{j \in J} \max_{i \in I}
\left\{ g(\omega_{-1},i,j)+p \cdot
([i(1-j)+(1-i)j] \right\},
 & \mbox{if} \ |p| \leq 2, \\
H_{-1}\left(2\frac{p}{|p|}\right)+|p|-2 & \mbox{if} \  |p| > 2.
\end{array}
\right.
\end{equation*}
For $x \in [-1,1]$ and $p \in \m{R}$, let
\begin{equation} \label{hamiltonian}
H(x,p):=|x| H_1(\left|p\right|)+(1-|x|)H_{-1}(\left|p\right|). 
\end{equation}
Note that the definition of $H_1$ and $H_{-1}$ for $|p| > 2$ ensures that $\lim_{|p| \rightarrow +\infty} H_1(p)=\lim_{|p| \rightarrow +\infty} H_{-1}(p)=+\infty$, thus
$\lim_{|p| \rightarrow +\infty} H(p)=+\infty$. 
Note also that for all $x \in [-1,1]$, $H_1(x,.)$ is increasing on $[-2,2]$ and $H_{-1}(x,.)$ is decreasing on $[-2,2]$.
%\begin{equation}
%H_1(p)=
%\min_{j \in J} \max_{i \in I}
%\left\{ g(\omega_1,i,j)-p \cdot
%([i(1-j)+(1-i)j] \right\},
%\end{equation}
%and
%\begin{equation}
%H_{-1}(p)=\min_{j \in J} \max_{i \in I}
%\left\{  g(\omega_{-1},i,j)+p \cdot  [i(1-j)+(1-i)j]) \right\}.
%\end{equation}
\\
Thanks to Proposition \ref{shapley} \ref{shapleyu} and Proposition \ref{propsto} \ref{prop2}, we have $\lambda u_\lambda(\omega_1)+H_1(u_\lambda(\omega_1)-u_\lambda(\omega_{-1}))=0$ and 
$\lambda u_\lambda(\omega_{-1})+H_{-1}(u_\lambda(\omega_1)-u_\lambda(\omega_{-1}))=0$.
\\
\\
For $x \in [-1,1]$, let $u_{\lambda}(x)=|x| u_{\lambda}(\omega_1)+(1-|x|) u_{\lambda}(\omega_{-1})$. Let $x \in (-1,1) \setminus \left\{0\right\}$. Proposition \ref{propsto} \ref{prop2} implies that $w_{\lambda}(\omega_{-1}) \leq w_{\lambda}(\omega_1)$, thus $u_{\lambda}(\omega_{-1}) \leq u_{\lambda}(\omega_1)$ and $|Du_{\lambda}(x)|=u_{\lambda}(\omega_1)-u_{\lambda}(\omega_{-1})$. 
Consequently, Proposition \ref{shapley} \ref{shapleyu} yields %for all $x \in (-1,1) \setminus \left\{0\right\}$,
\begin{equation} \label{HJBex}
\lambda u_\lambda(x)+H(x,Du_\lambda(x))=0.
\end{equation}
Note that the above equation is identical to equation (\ref{HJBex}). The reason why we use the notation $u_\lambda$ and not $v_\lambda$ is that, as we shall see, $c(H)=0$, thus $u_\lambda$ coincides with $v_\lambda$.

Extend $u_\lambda$ and $H(.,p)$ ($p \in \m{R}$) as 2-periodic functions defined on $\m{R}$. The Hamiltonian $H$ is continuous and coercive in the momentum, and the above equation holds in a classical sense for all $x \in \m{R} \setminus \m{Z}$. 
%Define
%\begin{equation*}
%H(x,p):=x H_1(p)+(1-x) H_{-1}(p),
%\end{equation*}
%and
%\begin{equation*}
%u_\lambda(x)=x u_\lambda(\omega_1)+(1-x) u_\lambda(\omega_{-1}).
%\end{equation*}
%Then $u_\lambda$ satisfies the discounted Hamilton-Jacobi equation
%\begin{equation}
%\lambda u_\lambda(x) +H(x,Du_\lambda(x))
%\end{equation}
%The only difference with the discounted Hamilton-Jacobi equation of the previous section is that the equation is defined on the simplex $\Delta(K)$, and not on a torus. 
%\\
%$H$ is obviously continuous and Lipschitz in $p$. Moreover, $(u_\lambda)$ is 1-Lipschitz. 
%\\
%On peut étendre $u_\lambda$ sur $\m{R}$ en une fonction 2-périodique, et de même pour $H(.,p)$; de telle sorte que $u_\lambda: \m{R}\rightarrow \m{R}$ soit solution sur $\m{R}$ privé de $\m{Z}$ de l'équation de Hamilton-Jacobi ci-dessus, au sens classique. Montrons que $u_\lambda$ est solution au sens de viscosité sur $\m{R}$ tout entier. Par 2-périodicité, il suffit de montrer qu'elle est solution au sens de viscosité en $x=0$ et $x=1$. 
\\
For $x \in \m{R}$, denote by $D^+u_\lambda(x)$ (resp., $D^- u_\lambda(x)$) the super-differential (resp., the sub-differential) of $u_\lambda$ at $x$. 
Let us show that $u_{\lambda}$ is a viscosity solution of (\ref{HJBex}) on $\m{R}$. By 2-periodicity, it is enough to show that this is a viscosity solution for $x=0$ and $x=1$. 
\\
Let us start by $x=0$. We have $D^+ u_{\lambda}(0)=\emptyset$ and $D^- u_{\lambda}(0)=[u_\lambda(\omega_{-1})-u_\lambda(\omega_{1}),u_\lambda(\omega_{1})-u_\lambda(\omega_{-1})]$. 
\\
Let $p \in D^- u_{\lambda}(0)$. Then $H_{-1}(p) \geq H_{-1}(u_\lambda(\omega_{1})-u_\lambda(\omega_{-1}))=-\lambda u_{\lambda}(\omega_{-1})$, thus $\lambda u_{\lambda}(0)+H(0,p) \geq 0$. Consequently, $u_\lambda$ is a viscosity solution at $x=0$. 
%Let $\phi$ be a $C^1$ function such that 
%$\phi(0)=u_\lambda(0)$ and $\phi \geq u$ on a neighbourhood of 0. Thus, $\left|D\phi(0)\right| \geq u_\lambda(\omega_1)-u_\lambda(\omega_{-1})$. It follows that $H_{-1}(\left|D\phi(0)\right|) \geq H_{-1}(u_\lambda(\omega_1)-u_\lambda(\omega_{-1}))=\lambda u_\lambda(\omega_{-1})=\lambda \phi(0)$. Consequently, $\lambda \phi(0)+H(0,D\phi(0)) \leq 0$. Thus, $u_\lambda$ is a viscosity sub-solution at 0. Likewise,  On $u_\lambda$ is a viscosity super-solution at 0. 
\\
\\
Consider now the case $x=1$. We have $D^+ u_{\lambda}(1)=[u_\lambda(\omega_{-1})-u_\lambda(\omega_{1}),u_\lambda(\omega_{1})-u_\lambda(\omega_{-1})]$ and $D^- u_{\lambda}(1)=\emptyset$. 
\\
Let $p \in D^+ u_{\lambda}(1)$. Then $H_{1}(p) \leq H_{1}(u_\lambda(\omega_{1})-u_\lambda(\omega_{-1}))=-\lambda u_{\lambda}(\omega_1)$, thus $\lambda u_{\lambda}(1)+H(1,p) \geq 0$. Consequently, $u_\lambda$ is a viscosity solution at $x=1$. 
\\
Let us now conclude the proof of Theorem \ref{main}. Because $H$ is 2-periodic, equation $(\ref{HJBex})$ can be considered as written on $\m{T}^1$. 

As noticed before, equation (\ref{HJBex}) is identical to equation (\ref{disceq}). Therefore, as stated in the introduction, $-\lambda u_{\lambda}$ converges to $c(H)$.
Proposition \ref{propsto} \ref{key} implies that $(-\lambda_n u_{\lambda_n}(1))$ converges to 0, thus $c(H)=0$. Still by Proposition \ref{propsto} \ref{key}, $(u_\lambda(1))$ does not have a limit when $\lambda$ tends to 0: Theorem \ref{main} is proved.

\section{Proof of Proposition \ref{propsto}} \label{proofprop}
\subsection{Proof of \ref{prop2}}
%Consider the stochastic game $\Gamma'$, with action set $I$ for Player 1, $J$ for Player 2, and payoff function 
%$-g$. In this game, Player 1 faces the same challenge as Player 2 in $\Gamma'$
Consider Proposition \ref{shapley} \ref{shapleyv} for $r=1$. Take $j=1/2 \in J$. It yields
\begin{eqnarray}
\nonumber
w_{\lambda}(\omega_1) &\leq& \max_{i \in I} \left\{ 1+(1-\lambda) \left(\frac{1}{2}w_{\lambda}(\omega_1)+\frac{1}{2} w_{\lambda}(\omega_{-1}) \right) \right\}
\\
&=& 1+\frac{1}{2}(1-\lambda)\left(w_{\lambda}(\omega_1)+w_{\lambda}(\omega_{-1}) \right)
\label{prog1}.
\end{eqnarray}
Take $i=1/2$. This yields
\begin{equation} \label{prog2}
w_{\lambda}(\omega_1) \geq \frac{1}{2}+\frac{1}{2}(1-\lambda)\left(w_{\lambda}(\omega_1)+w_{\lambda}(\omega_{-1}) \right).
\end{equation}
For $r=-1$, taking $j=1/2$ and then $i=1/2$ produce the following inequalities:
\begin{equation} \label{prog3}
w_{\lambda}(\omega_{-1}) \leq -\frac{1}{2}+\frac{1}{2}(1-\lambda)\left(w_{\lambda}(\omega_1)+w_{\lambda}(\omega_{-1}) \right),
\end{equation}
and
\begin{equation} \label{prog4}
w_{\lambda}(\omega_{-1}) \geq -1+\frac{1}{2}(1-\lambda)\left(w_{\lambda}(\omega_1)+w_{\lambda}(\omega_{-1}) \right).
\end{equation}
Combining (\ref{prog2}) and (\ref{prog3}) yield $w_{\lambda}(\omega_1) \geq w_{\lambda}(\omega_{-1})+1 \geq w_{\lambda}(\omega_{-1})$. 
Combining (\ref{prog1}) and (\ref{prog4}) yield $w_{\lambda}(\omega_{-1}) \geq w_{\lambda}(\omega_{1})-2$, and \ref{prop2} is proved. 

\subsection{Proof of \ref{key}}
%The result \ref{key} is useful in the proof of \ref{prop3}. Consequently, the proof of  \ref{key} is presented before the one of \ref{prop3}. 
%\\
For $(i,i') \in \left\{0,1\right\}^2$, consider the strategy $y$ of Player 1 that plays $i$ in $\omega_1$ and $i'$ in $\omega_{-1}$ (regardless of Player 2's actions), and the strategy $z$ of Player 2 that plays $a$ in state $\omega_1$, and $b$ in state $\omega_{-1}$. Denote $\gamma_{\lambda}^{i,i'}(a,b):=\gamma_{\lambda}^{\omega_1}(y,z)$ (resp., $\widetilde{\gamma}_{\lambda}^{i,i'}(a,b):=\gamma_{\lambda}^{\omega_{-1}}(y,z)$),  the payoff in $\Gamma_{\lambda}^{\omega_1}$ (resp., $\Gamma_{\lambda}^{\omega_{-1}}$), when $(y,z)$ is played. 

\begin{proposition}
The following hold: 
\begin{enumerate}
\item
\begin{equation*}
\gamma^{0,0}_{\lambda}(a,b)=\frac{-2(a-b-\lambda+b \lambda)}{\lambda(a+b+\lambda-a \lambda -b \lambda)}
\end{equation*}
\begin{equation*}
\gamma^{1,1}_{\lambda}(a,b)=-\frac{a-b+\lambda b}{\lambda(a+b+\lambda-a \lambda -b \lambda-2)}
\end{equation*}
\begin{equation*}
\gamma^{1,0}_{\lambda}(a,b)=\frac{2a+2b+2\lambda-a b -a \lambda-2b\lambda+a b \lambda-2}{\lambda(b-a+\lambda a -b \lambda+1)}
\end{equation*}
\begin{equation*}
\gamma^{0,1}_{\lambda}(a,b)=-\frac{2a+2b-a b -2 b \lambda+ a b \lambda-2}{\lambda(a-b-a \lambda +b \lambda +1)}
\end{equation*}
\item
\begin{itemize}
\item
$\gamma^{0,0}_{\lambda}$ is decreasing with respect to $a$ and increasing with respect to $b$. 
\item
$\gamma_{\lambda}^{1,1}$ is increasing with respect to $a$ and decreasing with respect to $b$.
\item
$\gamma_{\lambda}^{1,0}$ is increasing with respect to $a$ and $b$.
\item
$\gamma_{\lambda}^{0,1}$ is decreasing with respect to $a$ and $b$. 
\end{itemize}
\end{enumerate}
\end{proposition}
\begin{proof}
\begin{enumerate}
\item
The payoffs $\gamma^{0,0}_{\lambda}(a,b)$ and $\widetilde{\gamma}_{\lambda}^{0,0}(a,b)$ satisfy the following recursive equation:
\begin{eqnarray*}
\gamma^{0,0}_{\lambda}(a,b)&=&a(1-\lambda)\widetilde{\gamma}_{\lambda}^{0,0}(a,b)+(1-a) (2+(1-\lambda) \gamma^{0,0}_{\lambda}(a,b))
\\
\widetilde{\gamma}^{0,0}_{\lambda}(a,b)&=&a(1-\lambda) \gamma_{\lambda}^{0,0}(a,b)+(1-a) (-2+(1-\lambda) \widetilde{\gamma}^{0,0}_{\lambda}(a,b))
\end{eqnarray*}
Combining these two relations give the first equality. The three other equalities can be derived in a similar fashion. 
\item
These monotonicity properties are simply obtained by deriving $\gamma^{i,i'}_\lambda$ with respect to $a$ and $b$. 
%\item
%Development of $\gamma_{\lambda}^{0,0}(c_0+c \lambda, c_0+d \lambda)...$
\end{enumerate}
\end{proof}
For $\lambda \in (0,1]$, set $\displaystyle p^*(\lambda):=2-\sqrt{2}+\left(\frac{3}{4}-\frac{1}{\sqrt{2}}\right) \lambda$.
Define a strategy $y$ of Player 1 in the following way: 
\begin{itemize}
\item
in state $\omega_1$, play $0$ if $j \leq p^*(\lambda)$, play 1 otherwise,
\item
in state $\omega_{-1}$, play $1$ if $j \leq p^*(\lambda)$, play 0 otherwise. 
\end{itemize}
The rationale behind this strategy can be found in Section \ref{asympt}.  
\\
For all $n \geq 1$, define
\begin{equation*}
\lambda_n:=  \frac{2^{-2n}}{\displaystyle \frac{3}{4}-\sqrt{2}} \quad \text{and} \quad \mu_n:=\frac{2^{-2n-1}}{\displaystyle \frac{3}{4}-\sqrt{2}}. 
\end{equation*}
\begin{proposition}
The following hold:
\begin{enumerate}
\item
\begin{equation*}
\lim_{n \rightarrow +\infty} \min_{z \in Z} \gamma_{\lambda_n}(y,z)=\frac{1}{\sqrt{2}}
\end{equation*}
\item
\begin{equation*}
\lim_{n \rightarrow +\infty} \min_{z \in Z} \gamma_{\mu_n}(y,z)=\frac{5}{2 \sqrt{2}}-1>\frac{1}{\sqrt{2}}
\end{equation*}
\end{enumerate}
\end{proposition}
\begin{proof}
\begin{enumerate}
\item
For all $(i,i') \in \left\{0,1\right\}$, 
\begin{equation*}
\lim_{n \rightarrow +\infty} \gamma^{i,i'}_{\lambda_n}(p^*(\lambda_n),p^*(\lambda_n)) = \displaystyle \frac{1}{\sqrt{2}},
\end{equation*}
and the result follows. 
\item
Let $z$ be a strategy of Player 2, and $a=z(\omega_1)$ and $b=z(\omega_{-1}$). 
\\
Note that the interval $(p^*(\mu_n/2),p^*(2\mu_n))$ does not intersect $J$.

The following cases are distinguished:
\begin{case}{$a \leq p^*(\mu_n)$ and $b \leq p^*(\mu_n)$, thus $a \leq p^*(\mu_n/2)$ and $b \leq p^*(\mu_n/2)$}
\end{case}
We have $\gamma^{\omega_1}_{\mu_n}(y,z)=\gamma^{0,1}_{\mu_n}(a,b) \geq \gamma^{0,1}_{\mu_n}(p^*(\mu_n/2),p^*(\mu_n/2)) \underset{n \rightarrow +\infty}{\rightarrow} \displaystyle \frac{5}{4} \sqrt{2}-1$
\begin{case}{$a \leq p^*(\mu_n)$ and $b \geq p^*(\mu_n)$}, thus $a \leq p^*(\mu_n/2)$ and $b \geq p^*(2\mu_n)$
\end{case}
We have $\gamma^{\omega_1}_{\mu_n}(y,z)=\gamma^{0,0}_{\mu_n}(a,b) \geq \gamma^{0,0}_{\mu_n}(p^*(\mu_n/2), p^*(2\mu_n))  \underset{n \rightarrow +\infty}{\rightarrow} \displaystyle - \frac{1+2 \sqrt{2}}{8(-2+\sqrt{2})}$
\begin{case}{$a \geq p^*(\mu_n)$ and $b \leq p^*(\mu_n)$}, thus $a \geq p^*(2\mu_n)$ and $b \leq p^*(\mu_n/2)$
\end{case}
We have $\gamma^{\omega_1}_{\mu_n}(y,z)=\gamma^{1,1}_{\mu_n}(a,b) \geq \gamma^{1,1}_{\mu_n}(p^*(2\mu_n), p^*(\mu_n/2))  \underset{n \rightarrow +\infty}{\rightarrow} (-1/16) \displaystyle \frac{-25+14 \sqrt{2}}{\sqrt{2}-1}$
\begin{case}{$a \geq p^*(\mu_n)$ and $b \geq p^*(\mu_n)$}, thus $a \geq p^*(2\mu_n)$ and $b \geq p^*(2\mu_n)$
\end{case}
We have $\gamma^{\omega_1}_{\mu_n}(y,z)=\gamma^{1,0}_{\mu_n}(a,b) \geq \gamma^{1,0}_{\mu_n}(p^*(2\mu_n), p^*(2\mu_n))  \underset{n \rightarrow +\infty}{\rightarrow} \displaystyle -2+2\sqrt{2}$
\\
Among these cases, the smallest limit is $\frac{5}{4}(\sqrt{2}-1)$, and the result follows.
\end{enumerate}
\end{proof}
%\subsection{Proof of \ref{prop3}}
%It is well-known that the solution 
\section*{Acknowledgments}
The author is very grateful to Pierre Cardaliaguet, Andrea Davini, Abraham Neyman, Sylvain Sorin and Maxime Zavidovique for helpful discussions.

\bibliography{bibliogen}

\begin{thebibliography}{10}

\bibitem{BS00}
G.~Barles and P.E. Souganidis.
\newblock Some counterexamples on the asymptotic behavior of the solutions of
  hamilton--jacobi equations.
\newblock {\em Comptes Rendus de l'Acad{\'e}mie des Sciences-Series
  I-Mathematics}, 330(11):963--968, 2000.

\bibitem{DFIZ16}
A.~Davini, A.~Fathi, R.~Iturriaga, and M.~Zavidovique.
\newblock Convergence of the solutions of the discounted hamilton--jacobi
  equation.
\newblock {\em Inventiones mathematicae}, 206(1):29--55, 2016.

\bibitem{IS11}
C.~Imbert and S.~Serfaty.
\newblock Repeated games for non-linear parabolic integro-differential
  equations and integral curvature flows.
\newblock {\em Discrete Contin. Dyn. Syst}, 29(4):1517--1552, 2011.

\bibitem{KS06}
R.~Kohn and S.~Serfaty.
\newblock A deterministic-control-based approach motion by curvature.
\newblock {\em Communications on pure and applied mathematics}, 59(3):344--407,
  2006.

\bibitem{KS10}
R.~Kohn and S.~Serfaty.
\newblock A deterministic-control-based approach to fully nonlinear parabolic
  and elliptic equations.
\newblock {\em Communications on Pure and Applied Mathematics},
  63(10):1298--1350, 2010.

\bibitem{LPV86}
P-L. Lions, G.~Papanicolaou, and S.~Varadhan.
\newblock Homogenization of hamilton-jacobi equations.
\newblock {\em Unpublished preprint}, 1986.

\bibitem{MSZ}
J.F. Mertens, S.~Sorin, and S.~Zamir.
\newblock {\em Repeated games}.
\newblock CORE DP 9420-22, 1994.

\bibitem{vigeral13}
G.~Vigeral.
\newblock A zero-sum stochastic game with compact action sets and no asymptotic
  value.
\newblock {\em Dynamic Games and Applications}, 3(2):172--186, 2013.

\bibitem{Z13}
B.~Ziliotto.
\newblock Zero-sum repeated games: counterexamples to the existence of the
  asymptotic value and the conjecture maxmin= lim v (n).
\newblock {\em The Annals of Probability}, 44(2):1107--1133, 2016.

\bibitem{Z16}
B.~Ziliotto.
\newblock Stochastic homogenization of nonconvex hamilton-jacobi equations: a
  counterexample.
\newblock {\em Communications on Pure and Applied Mathematics},
  70(9):1798--1809, 2017.

\end{thebibliography}
\end{document}